\documentclass{amsart}
\usepackage{amsmath}
\usepackage{amsfonts}
\usepackage{amssymb,amscd,amsthm,amsbsy}
\usepackage{geometry}
\usepackage{mathrsfs}
\usepackage[bookmarksnumbered, colorlinks, linktocpage, plainpages]{hyperref}
\usepackage{nameref}
\usepackage{wasysym,amssymb,eufrak,indentfirst,color}
\usepackage{mathtools}
\usepackage{appendix}
\usepackage{extarrows}
\usepackage{setspace}

\usepackage{cases}
\usepackage{graphicx}
\usepackage{ragged2e}

\numberwithin{equation}{section}

\newcommand{\almost}{para-}

\newcommand{\Hom}{\text{Hom}}
\newcommand{\End}{\text{End}}

\newcommand{\bfs}{Without loss of generality we can assume}

\newcommand{\spf}{\mathbb{F}}
\newcommand{\spc}{\mathbb{C}}
\newcommand{\spr}{\mathbb{R}}
\newcommand{\spo}{\mathbb{O}}

\newcommand\huaa[1]{\mathscr{A}(#1)}
\newcommand{\re}{\text{Re}\,}%

\newcommand\fx[2]{\left<#1,#2\right>}%

\newcommand\fsh[1]{\left|\left|#1\right|\right|}%

\def\S{\mathbb{S}}
\def\O{\mathbb{O}}
\def\R{\mathbb{R}}

\def\abs#1{\left|#1\right|}

\newtheorem{mydef}{Definition}[section]
\newtheorem{rem}[mydef]{Remark}
\newtheorem{eg}[mydef]{Example}

\newtheorem{lemma}[mydef]{Lemma}
\newtheorem{thm}[mydef]{Theorem}

\newtheorem{step }[stp]{Step }

\begin{document}
	\title[Octonionic  isometry]{Octonionic isometric isomorphisms and partial isometry}	
	\author{Qinghai Huo}
	
	\email[Q.~Huo]{hqh86@mail.ustc.edu.cn}
	\address{School of Mathematics, Hefei University of Technology, Hefei 230601, China}
	
	\author{Guangbin Ren}
	\email[G.~Ren]{rengb@ustc.edu.cn}
	\address{School of Mathematics Sciences, University of Science and Technology of China, Hefei 230026, China}
	
	\author{Zhenghua Xu}
	\email[Z.~Xu]{zhxu@hfut.edu.cn}
	\address{School of Mathematics, Hefei University of Technology, Hefei 230601, China}
	\date{}
	
	\thanks{This work was supported by the National Natural Science Foundation of China (Nos.12571090,12301097), the Fundamental Research Funds for the Central Universities (Nos.JZ2025HGTB0171, JZ2025HGTG0250),  Anhui Provincial Natural Science Foundation (No.
		2308085MA04), and  China Scholarship Council (Nos. 202506690055, 202506690052). }

	%
	
	\begin{abstract} 
		
		Very recently, two new notions of para-linear mappings and weak associative orthonormal bases were introduced in octonionic functional analysis, which have been proved to be powerful in formulating the basic theory, such as the Riesz representation theorem and the Parseval theorem. In this article, we continue exploring more properties of these two concepts and initiate the study of octonionic para-linear isometric operators. Surprisingly, it is proven that the condition of the para-linear operator on a Hilbert octonionic bimodule being an isometric isomorphism is equivalent to  it  mapping any associative orthonormal basis to a weak associative orthonormal basis, which implies also that an octonionic matrix is an isometry if and only if the system of its row vectors is a weak associative orthonormal basis. Furthermore, we introduce the concept of para-linear partial isometric operators and establish the aforementioned analogue in this new setting. Based on these facts, we can provide naturally a new viewpoint of James questions by modifying the definition of  {octonionic Stiefel space}.
	\end{abstract}
	
	\maketitle
	
	
	\section{Introduction}
	
	%

	The theory of octonionic Hilbert spaces is
	initiated    by 	Goldstine and Horwitz \cite{goldstine1964hilbert}   in 1964. They gave a definition of  a linear space over octonions. In additional to the usual postulates for an inner product, they yet  assumed that 
	\begin{eqnarray}
		\fx{px}{x}&=&p\fx{x}{x},\label{eq:1}\\
		\re\fx{px}{y}&=&\re (p\fx{x}{y})
	\end{eqnarray}
	for all vectors $x,y$ and octonions $p$. In a second paper \cite{goldstine1964hilbert2},  Goldstine and Horwitz  investigated the concept of a Hilbert space over finite dimensional associative algebra. The theory of their proceeding paper is shown to be a special case of the theory of this paper. Jamison in 1970 summarized the development of functional analysis over quaternions and octonions \cite{James1970QOfct}. The octonionic linear functionals are the main subject and only a  few results are obtained for octonionic setting. 
	Some developments on operator theory and C*-algebra theory  are obtained at the beginning of the 21st century by	Ludkovsky and Spr\"ossig  \cite{Ludkowski2011aacaspec,ludkovsky2007Spectral}. 	Compared with the  rich quaternionic  Hilbert space theory
	\cite{soffer1983quaternion,viswanath1971normal,Colombo2008funcalculus,Alpay2016JMP,Colombo2011ADVqevol,Colombo2019normal,ghiloni2013slicefct},
	the  octonionic functional theory  still has many blanks and difficulties.
	
	Recently,  two new phenomena  have emerged that are completely different from  the classical associative settings.
	One is that   the place of linear mappings   is replaced by para-linear mappings.
	The other is that the place of orthonormal bases should be replaced by
	weak associative orthonormal bases. It was found  that the  axiom \eqref{eq:1}  in  the definition of $\O$-Hilbert space given by Goldstine and Horwitz in  \cite{goldstine1964hilbert} is actually  not independent
	\cite{huoqinghai2022Riesz}.
	This leads to    an abstract definition of mappings arising
	from the octonionic inner products, which are called
	octonionic para-linear mappings  \cite{huoqinghai2022Riesz}.
	The  notion of para-linearity turns out to be   the non-associative counterpart of linearity. The octonionic Riesz representation theorem   holds  for para-linear functionals \cite{huoqinghai2022Riesz} naturally.
	Based on the notion of para-linearity, the structures of Hilbert left $\O$-modules and  $\O$-bimodules are clear \cite{huoqinghai2021tensor}. The notion of para-linearity also has some applications in octonionic theory, such as \cite{Colombo2023OctonionicMA,Kraußhar2022discoct,Krauhar2023WeylCP}.
	
	The notion of weak associative orthonormal bases comes from the investigation of the Parseval theorem. 	At first, Goldstine and Horwitz   \cite{goldstine1964hilbert2} observed that the  Parseval theorem can not hold for general orthonormal basis. What is more,   it turns out that the definition of the dimension of the octonionic Hilbert space  in the sense of the cardinality of  orthonormal basis  is meaningless,
	%
	since   the 	cardinalities may vary for different  orthonormal bases.
	For example, as observed in \cite{huoqinghai2021tensor}, the $\O$-Hilbert space $\O^2$ admits two orthonormal  bases $$(0,1),\quad(1,0)$$ and $$\dfrac{1}{\sqrt{2}}\Big((e_1,e_2),\quad (e_4,e_7),\quad (e_6,e_5),\quad (1,e_3)\Big).$$
	One has cardinality $2$ and the other has cardinality $4$.
	Recently, we introduce a notion of weak associative orthonormal basis, for which we establish   the  Parseval theorem  \cite{huoqinghai2021tensor}.

	In this article, we aim to characterize the para-linear isometric isomorphisms and partial isometries in Hilbert octonionic bimodule. 
	It is known that the   groups of orthogonal 	matrices $O_{\spf}(n)$  over $\spf$ for  $\spf=\R,\spc, \mathbb H$, i.e.,
	$$O(n), U(n),Sp(n)$$ respectively,   play   important roles in  the theory of Lie groups, homogeneous spaces, holonomy groups, and    Stiefel
	manifolds. 
	It is a natural question to ask what is the
	octonionic  orthogonal ``group''. As  a special finite dimensional case, we shall give  the characterization of octonionic  isometric matrix, which highly  relies on the notion of weak associative orthonormal bases.

	We introduce the notion of para-linearity now.
	An \textbf{$\O$-para-linear} function in a Hilbert $\O$-module $H$
	is defined to be a real linear map $f: H\to\O$ subject to the condition
	\begin{eqnarray}\label{eq:def-para-104}\re \big(f(px)-pf(x) \big)=0\end{eqnarray} for all $p\in \O$ and $x\in M$ (\cite{huoqinghai2022Riesz}).
	Similarly,   if  $\O$ is replaced by other division algebras $\spf$, then one can define
	the   notion of  $\spf$-\almost linear functions. It is shown that the
	$\spf$-\almost linearity becomes   $\spf$-linearity	when  $\spf$ is  associative.
	By introducing the real part structure of an $\O$-bimodule, we extend the notion of para-linear functions to the notion of para-linear maps with the range $\O$ replaced by arbitrary $\O$-bimodule (\cite{huoqinghai2020nonass}). In particular, every left para-linear operator on $\O^n$ is actually a map induced by an octonionic matrix acting on row vectors from the right side.

	A subset $S=\{x_{\alpha}\}_{\alpha\in \Lambda}$ of $H$ is said to be an \textbf{orthonormal set} if $$\fx{x_{\alpha}}{x_{\beta}}=\delta_{\alpha \beta}.$$
	Further if $S\subseteq \re {H}$, where $\re H$ is the real part of $H$, then  $S$ is said to be an \textbf{associative orthonormal   set}.	
	An orthonormal system $S=\{x_{\alpha}\}_{\alpha\in \Lambda}$ is said to be an \textbf{orthonormal basis} if there does not exist other orthonormal system $S'$ such that $S\subsetneqq S'$.
	Further if $S\subseteq \re {H}$, then  $S$ is said to be an \textbf{ associative orthonormal   basis}.
	An orthonormal basis (set) $S=\{x_{\alpha}\}_{\alpha\in \Lambda}$ is said to be a \textbf{weak associative orthonormal basis (set)} if  for all $\alpha,\beta\in \Lambda$,
	$$A_p(x_\alpha,x_\beta):=\fx{px_\alpha}{x_\beta}-p\fx{x_\alpha}{x_\beta}=0.$$
	The weak associative orthonormal basis admits the following features \cite{huoqinghai2021tensor}:
	
	$\bullet$ The octonionic  Parseval theorem holds for an orthonormal basis if and only if the basis  is a weak associative orthonormal basis.

	$\bullet$ In any octonionic Hilbert space, there always exists a  weak associative orthonormal basis.

	$\bullet$ For any two weak associative orthonormal bases, they share the same  cardinality.




	The para-linear isometric isometry can be described in terms of the weak  associative orthonormal basis.
	
	\begin{thm}
		Let  $T:H\to H$ be a para-linear operator. Then the following are equivalent.
		\begin{enumerate}
			\item $T$ is an  isometric isomorphism.
			\item For any associative orthonormal basis $\{\epsilon_\alpha\}_{\alpha\in \Lambda}$ of $H$,
			$\{T\epsilon_\alpha\}_{\alpha\in \Lambda}$  is a weak associative orthonormal basis of $H$.
			\item   For any associative orthonormal basis $\epsilon=\{\epsilon_\alpha\}_{\alpha\in \Lambda}$, there exists a unique weak associative  orthonormal basis $\xi=\{\xi_\alpha\}_{\alpha\in \Lambda}$ such that  
			$$Tx=\sum_{\alpha\in \Lambda} \fx{{x}}{{\xi_\alpha}}\epsilon_\alpha.$$
			\item There exists an associative orthonormal basis $\epsilon=\{\epsilon_\alpha\}_{\alpha\in \Lambda}$, such that 	$\{T\epsilon_\alpha\}_{\alpha\in \Lambda}$  is a weak associative orthonormal basis of $H$.
			
			\item  $T^*\circ T=T\circ T^*=\text{Id}$. Here the notation ``$\circ$'' stands for the ordinary composition.
		\end{enumerate}
		Moreover, if $T$ is an $\O$-para-linear isometric isomorphism, then   
		\begin{eqnarray}\label{eq:T*T=I}
			T^*\circledcirc T=T\circledcirc T^*=\text{Id}.
		\end{eqnarray} Here the notation ``$\circledcirc$'' stands for the regular composition. The regular composition is a new composition of para-linear operators under which the para-linearity can be preserved. 
		
	\end{thm}

	We next consider the partial isometric operators. A para-linear operator $T$ is called a \textbf{partial isometry} if $T|_{(\ker T)^\perp}$ is an isometry. 
	
	\begin{thm}\label{intthm:p.iso}
		Let $T$ be a para-linear operator on $H$. Suppose $\ker T$ is an $\O$-submodule of $H$.
		Then the following are equivalent.
		\begin{enumerate}
			\item $T$ is a partial isometry.
			\item  For any 	associative orthonormal  set $\{\epsilon_\alpha\}_{\alpha\in \Lambda}$ in $(\ker T)^\perp$, $\{T\epsilon_\alpha\}_{\alpha\in \Lambda}$ is a weak associative orthonormal  set of $H$.
			\item There exists an 	associative orthonormal  basis $\{\epsilon_\alpha\}_{\alpha\in \Lambda}$ of Hilbert $\O$-submodule $(\ker T)^\perp$, such that $\{T\epsilon_\alpha\}_{\alpha\in \Lambda}$ is a weak associative orthonormal  set of $H$.
			
		\end{enumerate}
		
	\end{thm}

	Let $H=\O^n$. Each para-linear operator $T$ on $H$ can be  viewed as  an octonionic matrix $\pmb{T}$ with respect  to some associative orthonormal basis. Then we get
	\begin{thm}\label{intthm:ft equv}
		The following are equivalent:
		\begin{enumerate}
			\item 	$\pmb{T}$ is an (partial) isometry.
			\item The row vectors  of $\pmb{T}$ form a weak  associative orthonormal basis (set) of $\O^n$.
		\end{enumerate}
		
	\end{thm}

	We remark that if $\pmb{T}$ is an isometry, then we have 	$$\pmb{TT^*}=\pmb{T^*T}=\pmb{I}.$$ However, different from the classical cases, the converse is false. Let $$\pmb{T}=\dfrac{1}{\sqrt{2}}\left(\begin{array}[c]{ccc}
		1&e_1 \\
		-e_3& e_2
	\end{array}\right).$$
	One can check that $\pmb{T}$ satisfies that $\pmb{TT^*}=\pmb{T^*T}=\pmb{I}$, but the row vectors of $\pmb{T}$ is not a weak associative orthonormal basis, which means that $\pmb{T}$ is not an isometry.

	{Inspired by these results, we give a modification to James questions.
	In 1958, I. M. James \cite{james1958stiefel} defined the octonionic Stiefel space $V_k(\O^n)$ as  the space of orthonormal $k$-frames in $\O^n$, i.e.,
	\begin{eqnarray}\label{eqdef:James}
		V_k(\O^n):=\{(x_1,\dots, x_k)\mid x_i\in \O^n,\fx{x_i}{x_j}_{\O}=\delta_{ij}, 1\leqslant i,j\leqslant k\}.
	\end{eqnarray}
	James conducted a preliminary study of the case $k = 2$ and posed two fundamental questions concerning the general structure of $V_k(\mathbb{O}^n)$:
 	\begin{enumerate}
			\item When is the canonical
			projection
		{$$\pi: V_k(\O^n)\to V_l(\O^n)  \qquad   (1\leqslant l< k\leqslant n)$$}
			a fiber map?
			\item  Is $	V_k(\O^n)$ a manifold?
	\end{enumerate}}
	
	 {When $k = 2$, James \cite{james1958stiefel} 
	 	proved that the projection $\pi:V_2(\O^n)\to V_1(\O^n) $  is indeed a fiber map, implying that $V_2(\mathbb{O}^n)$ has a manifold structure.  
	 	James  \cite{james1958stiefel} also found that a cross-section of the fiber bundle $\pi:V_2(\O^n)\to V_1(\O^n) $ occurs when $n = 240$. More recently, Qian,  Tang, and  Yan \cite{qianchao2022PJM,qianchao2022}  provided partial answers to these classical questions using the theory of isoparametric hypersurfaces \cite{Tangziz2014jfa,Tangziz2013jdg,Tangziz2015adv}.
	 In particular, they showed  \cite{qianchao2022PJM} that the sphere bundle 
	 $$S^{8n-9}\hookrightarrow V_2(\O^n)\to V_1(\O^n)$$
	 admits a cross-section if and only if $n$ can be divided by $240$. They also demonstrated that for $n > k > 2$, the projection $\pi: V_{k+1}(\mathbb{O}^n) \to V_k(\mathbb{O}^n)$ is not a fibration in the sense of Serre, thereby providing a partial answer to Question (1). Further insights into Question (2) are given in
	  \cite{qianchao2022}. These findings highlight significant differences between the octonionic and classical cases and help motivate the modified James questions considered in our paper. 
	  }
	 

	In view of Theorem \ref{intthm:ft equv}, it is natural to consider a new notion of  octonionic Stiefel  spaces as \begin{eqnarray*}
		V^{\text{w}}_k(\O^n)
		:=\{(x_1,\dots, x_k)\mid x_i\in \O^n,(x_1,\dots, x_k) \text{ is a weak associative orthonormal set}\}.
	\end{eqnarray*}
	The canonical \textbf{James Questions} can thus be modified   as follows:
	\begin{enumerate}
		\item\label{intit:ques1} When is the canonical
		projection
		$$\pi: V^{\text{w}}_k(\O^n)\to V^{\text{w}}_l(\O^n)  \qquad   (1\leqslant l< k\leqslant n)$$
		a fiber map?
		\item \label{intit:ques2} Is $	V^{\text{w}}_k(\O^n)$ a manifold?
	\end{enumerate}
	We  give a negative answer to  \textbf{Questions} \eqref{intit:ques1} when $k=2$. 
	
	We calculate $V^{\text{w}}_2(\O^2)$ concretely. It turns out that
	$$V^{\text{w}}_2(\O^2)\cong\bigcup_{J\in \S^6}\S^7\times U_{\spc_J}(2)/ \sim, $$ where $\sim$ is an equivalence relation defined by
	$$(p,\pmb{T})\sim(q,\pmb{S})\iff p\pmb{T}=q\pmb{S}.$$
	Here $\S^6$ is the set of imaginary units of $\O$  and
	$U_{\spc_J}(2):=U({\spc_J}^{2})$ is the unitary group of ${\spc_J}^{2}$.
	This indicates that $V^{\text{w}}_n(\O^n)$ may possess a book structure with each page a Moufang loop structure, rather than a group structure.
	
	Note that the non-associative algebra of octonions and the Moufang loop of unit octonions have been applied in  $G_2$ theory,  such as \cite{Grigorian2017octonionbundles,Grigorian2021advloop}. The study  of octonionic isometric operators may have potential applications in $G_2$ theory. 


	\section{Para-linear  isometric operators on $\O$-Hilbert spaces}
	The theory of octonionic Hilbert spaces is initiated by Goldstine and Horwitz in 1964 \cite{goldstine1964hilbert}. Recently, we find a redundant axiom on their definition. This leads to a new notion of $\O$-\almost linearity \cite{huoqinghai2022Riesz}.
	We shall study the para-linear isometric operators in this section.
	\subsection{Para-linear maps and $\spo$-Hilbert spaces}
	We review in this subsection some basic notions  and properties of para-linear maps  on $\O$-modules and $\spo$-Hilbert spaces.

	The octonions   $\spo$  are the non-associative, non-commutative, normed division algebra over  $\spr$. For convenience, we  denote  $ e_0=1$. Let     $e_0,e_1,\dots,e_7$ be a basis of $\O$ as a real vector space.
	The multiplication rules can be expressed as
	\begin{align}\label{eq:epsilon notation}
		&e_ie_j=\epsilon_{ijk}e_k-\delta_{ij}, \quad i,j=1,\dots,7,
	\end{align}
	where $\delta_{ij}$ is the Kronecker delta, and $\epsilon_{ijk}$ is a completely antisymmetric tensor with value 1 when $$ijk = 123, 145, 176, 246, 257, 347, 365.$$

	A  real vector space $M$ is called  a left  $\O$-module if it is endowed with an $\O$-scalar multiplication such that    the \textbf{left associator} is    left  alternative.
	That is, for all  $p,q\in\O$ and $  m\in M$ we have
	\begin{eqnarray}\label{eq:left ass}
		[p,q,m]=-[q,p,m],
	\end{eqnarray}
	where the left associator is defined by
	$[p,q,m]:=(pq)m-p(qm).$
	An element $m\in M$ is called an associative element if for all $p,q\in \O$, there holds
	$$[p,q,m]=0.$$ We denote by $\huaa{M}$ the set of associative elements of  $M$.
	A left  $\O$-module $M$ is called an \textbf{$\O$-bimodule} if there exists a right $\O$-scalar multiplication such that 
	$$[p,q,m]=[m,p,q]=[q,m,p]$$ for all $p,q\in \O$. Here 
	$${[m,p,q]:=(mp)q-m(pq)},\quad [q,m,p]:=(qm)p-q(mp).$$
	
	Let \( M \) be an $\O$-bimodule. There is a  decomposition of \( M \) as a real vector space: 
	\begin{equation}\label{eq:reM=huaaM=435}  M = \huaa{M} \oplus e_1\huaa{M} \oplus \cdots \oplus e_7\huaa{M}. \end{equation}
	The projection onto the first component of this decomposition, 
	\[ \re: M \rightarrow \huaa{M}, \]
	is termed the \textbf{real part operator} of \( M \) \cite{huoqinghai2020nonass}. For any $x\in M$, there is a decomposition
	$$x=\sum_{i=0}^7e_ix_i,$$ 
	where $x_i\in \re M$ for $i=0,\dots,7$.
	We can express the real part operator explicitly as follows \cite{huoqinghai2020nonass}  
	\[ \re x = \frac{5}{12}x - \frac{1}{12} \sum_{i=1}^7 e_i x e_i. \]
	For more details about  the real part operator, we refer   to \cite{huoqinghai2020nonass}.
	For details on the real part structure in the context of quaternions, refer to \cite{ng2007quaternionic}.

	\medskip
	
	We now recall the definition of para-linear maps.
	
	\begin{mydef}[\cite{huoqinghai2022Riesz}] Let $M$ be a left $\O$-module and $M'$ be an $\O$-bimodule.
		A real linear map $f: M \xlongrightarrow[]{} M'$
		is called an \textbf{$\O$-\almost linear} map if it satisfies
		$$\re A_p(x,f)=0$$
		for any $p\in\O$ and $x\in M$.
		Here $$A_p(x,f):=f(px)-pf(x),$$ which is called the \textbf{second associator} related to $f$.
		
	\end{mydef}	
	
	\begin{rem}
		We change the notation of second associator which is denoted by $[p,x,f]$ in \cite{huoqinghai2022Riesz,huoqinghai2021tensor}. This is because the position of $x$ and $f$  are equivalent in view of Riesz representation theorem. And we shall use several kinds of associators many times in the sequel. To  avoid confusion, we adapt the new notation $A_p(x,f)$ which keeps the order of $p,x,f$. This is consistent with the notation in \cite{huoqinghai2020nonass}.
	\end{rem}
	
	The following theorem will be used frequently in the sequel.
	\begin{thm}[\cite{huoqinghai2020nonass}]\label{cor: T re=0 iff T=0}
		Let $M,M'$ be two  $\O$-bimodules and $f:M\to M'$ be a real linear map.   Then $f$ is para-linear if and only if
		\begin{eqnarray}
			A_p(x,f)=\sum_{i=1}^7
			e_i \re f([e_i,p,x]).
		\end{eqnarray}
		In particular, the following hold.
		\begin{enumerate}
			\item For any $x\in \re M$ and $p\in \O$,
			we have $A_p(x,f)=0$.
			\item $f=0$ if and only if $f|_{\re M}=0$.
		\end{enumerate}
		
	\end{thm}
	
	We next introduce the regular composition of para-linear operators.

	{	\begin{mydef}[\cite{huoqinghai2020nonass}]\label{def:regular composition}	
		Given two left $\spo$-modules $M$ and $M'$, and an $\spo$-bimodule $M''$. If $f \in \Hom_{\mathcal{LO}}(M',M'')$ and $g \in \Hom_{\spr}(M,M')$, the \textbf{(left) regular composition} of $f$ and $g$ is the left para-linear map $f\circledcirc g$ defined as:
	\begin{equation}\label{def:regular composition-038}
			(f\circledcirc g)(x) := f(g(x))+[f,g,x].
		\end{equation}
		where   
		$$[f,g,x]:=-\sum_{i=0}^7 e_i \re \Big( f \big(A_{e_i} (x,g)\big) \Big).$$
	\end{mydef}}
	
	\begin{lemma}[\cite{huoqinghai2020nonass}]\label{lem:vanishing[f,g,x]=0}
		Let $M$ and $M'$ be two left $\spo$-modules, and let $M''$ be an $\spo$-bimodule. The relation $[f,g,x]=0$ holds, yielding the equality: $$f\circledcirc g(x)=f (g(x))$$ 
		under any of the following scenarios:
		\begin{enumerate}
			\item $g\in \Hom_{\spo}(M,M')$;
			\item $M'$ is an $\O$-bimodule, $x\in \huaa{M}$,  and  $g\in \Hom_{\mathcal{LO}} (M,M')$;
			\item $M'$ is an $\O$-bimodule, $f\in  \Hom_{\spo}(M',M'')$,    and  $g\in \Hom_{\mathcal{LO}} (M,M')$,
		\end{enumerate}
	\end{lemma}

	\begin{eg}\label{eg:paralinear}
		It is clear that $\O^n$ is an $\O$-bimodule and the real part of $\O^n$ is $\R^n$.
		Consider a matrix 
		$$\pmb{T}:=\left(\begin{array}[c]{ccc}
			t_{ij}
		\end{array}\right)_{n\times n}\in \End(\O^n).$$
		Then the map 
		\begin{eqnarray*}
			f_{\pmb{T}}:\O^n&\to& \O^n\\
			(x_1,\dots,x_n)&\mapsto& \left(\sum_{i=1}^nx_it_{i1},\dots,\sum_{i=1}^nx_it_{in}\right)
		\end{eqnarray*}
		is  a left $\O$-para-linear operator.
		The second associator of $f_{\pmb{T}}$ is 
		$$A_p((x_1,\dots,x_n),f_{\pmb{T}})=\left(\sum_{i=1}^n[p,x_i,t_{i1}],\dots,\sum_{i=1}^n[p,x_i,t_{in}]\right)
		$$	for any $p\in \O$. 	One can see that there is a bijection between the set of   left para-linear operators on $\O^n$ and octonionic matrices on $\O^n$.

		Let 		$\pmb{S}:=\left(\begin{array}[c]{ccc}
			s_{ij}
		\end{array}\right)_{n\times n}\in \End(\O^n)$ be another octonionic matrix and denote by $f_{\pmb{S}}$ the left para-linear operator induced by $\pmb{S}$. Then for any $(x_1,\dots,x_n)\in \re \O^n=\R^n$,
		\begin{align*}
			(f_{\pmb{S}}\circledcirc f_{\pmb{T}})((x_1,\dots,x_n))&=(f_{\pmb{S}}\circ f_{\pmb{T}})(x_1,\dots,x_n)\\
			&=\left(\sum_{i,j=1}^n(x_it_{ij})s_{j1},\dots,\sum_{i,j=1}^n(x_it_{ij})s_{jn}\right)\\
			&=\left(\sum_{i,j=1}^nx_i(t_{ij}s_{j1}),\dots,\sum_{i,j=1}^nx_i(t_{ij}s_{jn})\right)
		\end{align*}
		Let $\pmb{K}:=\pmb{T}\pmb{S}$ be the matrix obtained by the octonionic multiplication of octonionic matrices and $f_{\pmb{K}}$ be the para-linear induced by the matrix $\pmb{K}$.
		Then $$f_{\pmb{K}}x=(f_{\pmb{S}}\circledcirc f_{\pmb{T}})(x)$$ for any $x\in \re \O^n$.
		It follows from Theorem \ref{cor: T re=0 iff T=0} that $f_{\pmb{S}}\circledcirc f_{\pmb{T}}=f_{\pmb{K}}$.
		This means that the regular composition of two left para-linear operators induced by octonionic matrices is the para-linear operator induced by the opposite multiplication of octonionic matrices.

	\end{eg}
	
	\bigskip

	The theory of $\O$-Hilbert spaces is initiated by Goldstine and Horwitz in 1964 \cite{goldstine1964hilbert}. 
	The definition of $\O$-Hilbert spaces can be restated in terms of \almost linearity.
	\begin{mydef}[\cite{huoqinghai2022Riesz}]\label{def:hilbert O space}
		A left $\spo$-module $H$ is called a \textbf{ pre-Hilbert  $\spo$-module} if there exists an $\spr $-bilinear map $\left\langle\cdot,\cdot \right\rangle :H\times H \rightarrow \O$, which is referred to as an \textbf{$\spo$-inner product}, satisfying:
		\begin{enumerate}
			\item \textbf{($\O$-\almost linearity)} $\left\langle\cdot,u\right\rangle$ is (left) $\spo$-\almost linear for all $u\in H$.
			\item \textbf{(Octonion hermiticity)} $\left\langle u ,v\right\rangle=\overline{\left\langle v ,u\right\rangle}$ for all $ u,v\in H$.
			\item \textbf{(Positivity)} $\left\langle u ,u\right\rangle\in \spr^+$ and $\left\langle u ,u\right\rangle=0$ if and only if $u=0$.
		\end{enumerate}
		A pre-Hilbert left $\spo$-module	$H$ is said to be a \textbf{Hilbert left $\spo$-module} if it is complete with respect to its natural distance  induced by the norm  $\fsh{\cdot}:H\rightarrow \mathbb{R}^+ $
		\begin{eqnarray}\label{def:fshu ||u||}
			\fsh{u}=\sqrt{\fx{u}{u}}.
		\end{eqnarray}
	\end{mydef}
	
	We denote  $$\left\langle x,y\right\rangle _{\spr}:=\re\left\langle x,y\right\rangle$$ in this article.
	An important relation of the $\O$-inner product with $\fx{\cdot}{\cdot}_\R$ is
	\begin{eqnarray}\label{eq:<>=sum <>R}
		\fx{u}{v}=\sum_{i=0}^7e_i\fx{\overline{e_i}u}{v}_\R.
	\end{eqnarray}

	Denote  $$A_p(u,v):=\left\langle pu ,v\right\rangle-p\left\langle u ,v\right\rangle$$ for $u,v\in H$ and $p\in \O$, and  call $A_p(u,v)$  \textbf{the second associator} of $H$.
	We collect some identities for  the second associators of $H$.

	\begin{lemma}\label{lem:the second associator of H 's prop}
		For all $u,v\in H$ and all $p,q\in \spo$, the following  hold.
		\begin{eqnarray}
			A_p(u,v)&=&-A_p(v,u);\label{eq:Ap(u,v)=-Ap(v,u)}\\
			\fx{u}{pv}&=&\fx{u}{v}\overline{p}+A_p(u,v);\label{eq:<u,pv>=<u,v>p^+Ap(uv)}\\
			\fx{pu}{qv}&=&(p\fx{u}{v})\overline{q}+A_{pq}(u,v)+\fx{[p,q,v]}{u}.\label{eq:<pu,pv>in lemma}
		\end{eqnarray}
	\end{lemma}
	
	%
	If an $\O$-bimodule admits a complete  $\O$-inner product, then it is called a \textbf{Hilbert $\O$-bimodule}.
	It turns out that a  Hilbert $\O$-bimodule $H$ is actually the tensor product of its real part $\re H $  and the algebra of octonions $\O$; see \cite{huoqinghai2021tensor}.
	Indeed, for any associative elements $u,v\in \re H$, we have 
	$$\fx{u}{v}\in \R.$$
	Hence for any  $x,y\in H$ with decompositions
	$$x=\sum_{i=0}^7e_ix_i,\quad y=\sum_{i=0}^7e_iy_i,$$  where $x_i,y_i\in \re H$ for $i=0,\dots ,7$, we have
	\begin{eqnarray}\label{eq:tensor pd}
		\fx{x}{y}=\sum_{i,j=0}^7e_i\overline{e_j}\fx{x_i}{y_j}.
	\end{eqnarray}

	We next give some preliminary discussion of dual operation.
	Let $T$ be a para-linear operator defined on a Hilbert $\O$-bimodule $H$. Denote by $T^*$ the real  dual operator of $T$, in the sense of  regarding $T$ as a real linear operator of the real Hilbert space $(H,\fx{\cdot}{\cdot}_{\spr})$. That is, for any $x,y\in H$, 
	$$\fx{Tx}{y}_{\spr}=\fx{x}{T^*y}.$$
	\begin{lemma}\label{lem:T* is paralin}
		Let $T$ be a para-linear operator defined on a Hilbert $\O$-bimodule $H$. Then $T^*$ is also para-linear. Moreover, we have
		\begin{equation}\label{eq:<x,T*y>=<TX,y>+[yTx]}
			\fx{x}{T^*y} = \fx{Tx}{y}-\sum_{i=1}^7 e_i\fx{A_{e_i}(x,T)}{y}_{\spr}. 
		\end{equation}
		In particular, if one of $x$ and $y$ is in $\re H$, then 
		\begin{eqnarray}\label{eq:<xTy>=<T*xy>}
			\fx{x}{T^*y} = \fx{Tx}{y}.
		\end{eqnarray}
	\end{lemma}
	
	\begin{proof}
		We first establish the formula \eqref{eq:<x,T*y>=<TX,y>+[yTx]}.
		In view of \eqref{eq:<>=sum <>R}, we have
		\begin{eqnarray*}
			\fx{x}{T^*y} &=&\sum_{i=0}^7e_i\fx{\overline{e_i}x}{T^*y}_\R\\
			&=&\sum_{i=0}^7e_i\fx{T(\overline{e_i}x)}{y}_\R\\
			&=&\sum_{i=0}^7e_i\fx{\overline{e_i}Tx+A_{\overline{e_i}}(x,T)}{y}_\R\\
			&=&\fx{Tx}{y}-\sum_{i=1}^7e_i\fx{A_{{e_i}}(x,T)}{y}_\R.
		\end{eqnarray*}
		This proves \eqref{eq:<x,T*y>=<TX,y>+[yTx]}. If $x\in \re H$, then by definition we have  $A_{{e_i}}(x,T)=0$ and thus \eqref{eq:<xTy>=<T*xy>} holds. If $y\in \re H$, noticing that $\re A_{{e_i}}(x,T)=0$, then it follows from \eqref{eq:tensor pd} that $$\fx{A_{{e_i}}(x,T)}{y}_\R=0$$ and thus \eqref{eq:<xTy>=<T*xy>} holds.
		
		We next show that $T^*$ is para-linear. By definition, we need to show that for any $x\in H$ and $p\in\O$, $$\re A_p(x,T^*)=0.$$
		Let $y\in \re H$ be an arbitrary associative element. It follows from \eqref{eq:<x,T*y>=<TX,y>+[yTx]} that
		\begin{eqnarray}\label{pfeq:1}
			\fx{A_p(x,T^*)}{y}&=&\fx{T^*(px)-pT^*x}{y}\\
			&=&\fx{px}{Ty}-p\fx{T^*x}{y}\notag\\
			&=&\fx{px}{Ty}-p\fx{x}{Ty}\notag\\
			&=&A_p(x,Ty).\notag
		\end{eqnarray}
		Suppose $$A_p(x,T^*)=\sum_{i=0}^7e_iz_i$$ with $z_i\in \re H$ for $i=0,\dots, 7$. We claim that 
		$$\re A_p(x,T^*)=z_0=0.$$
		Indeed, since $y\in \re H$, we can conclude from \eqref{pfeq:1} that 
		$$0=\re\fx{A_p(x,T^*)}{y}=\re \fx{\sum_{i=0}^7e_iz_i}{y}=\re\fx{z_0}{y}.$$
		Note that $y$ is  an arbitrarily fixed associative element. This implies that $z_0=0$ as desired.
	\end{proof}


	\section{Octonionic para-linear isometric operator}

	We first give some preliminary  discussions on  the para-linear isometric operators.
	We  introduce the notion of $\O$-isometry on Hilbert $\O$-bimodules as follows.

	\begin{mydef}
		Let $H_1, H_2$ be two Hilbert $\O$-bimodules. A para-linear  operator $T:H_1\to H_2 $ is called an \textbf{isometry} if $T$ satisfies
		$$\fsh{Tx}=\fsh{x}$$ for all $x$.
	\end{mydef}
	
	In contrast to the complex or quaternionic case, an $\O$-para-linear isometry $T$ of Hilbert $\O$-bimodules  may fail to preserve the $\O$-inner product.
	\begin{lemma}\label{lem:isome}
		Let $T:H_1\to H_2 $ be an $\O$-isometry of $\O$-Hilbert spaces.
		Then we have
		\begin{eqnarray}\label{eq:<Tx,Ty>=<xy>+[]}
			\fx{Tx}{Ty}=\fx{x}{y}+\sum_{i=1}^7e_i\fx{ A_{e_i}(x,T)}{Ty}_\R
		\end{eqnarray}
		for all $x,y\in H_1$.
		In particular, for any $x\in \re H_1$, $y\in H_1$, we have
		\begin{eqnarray}\label{eq:<Tx,Ty>=<xy>}
			\fx{Tx}{Ty}=\fx{x}{y}.
		\end{eqnarray}
		
	\end{lemma}
	\begin{proof}
		Since $T$ is also an isometry between  real Hilbert spaces $(H_1,\fx{\cdot}{\cdot}_\R)$ and $(H_2,\fx{\cdot}{\cdot}_\R)$, it follows that
		$$\fx{Tx}{Ty}_\R=\fx{x}{y}_\R.$$
		Thus according  to \eqref{eq:<>=sum <>R}, we obtain
		\begin{eqnarray*}
			\fx{Tx}{Ty}&=&\fx{Tx}{Ty}_\R-\sum_{i=1}^7e_i\fx{e_i Tx}{Ty}_\R\\
			&=&\fx{x}{y}_\R-\sum_{i=1}^7e_i\fx{ T(e_ix)-A_{e_i}(x,T)}{Ty}_\R\\
			&=&\fx{x}{y}_\R-\sum_{i=1}^7e_i\fx{ e_ix}{y}_\R+\sum_{i=1}^7e_i\fx{ A_{e_i}(x,T)}{Ty}_\R\\
			&=&\fx{x}{y}+\sum_{i=1}^7e_i\fx{ A_{e_i}(x,T)}{Ty}_\R.
		\end{eqnarray*}
		This proves the lemma.
	\end{proof}
	\begin{rem}\label{rem:preserve inner prod not work}
		In view of   Lemma \ref{lem:isome}, one can see that the condition for an operator $T$ preserving the $\O$-inner product is a very strong condition.
		If an operator $T$ satisfies that $$\fx{Tx}{Ty}=\fx{x}{y}$$ for all $x,y\in H$ and suppose  $T$ is surjective, then \eqref{eq:<Tx,Ty>=<xy>+[]} implies that
		$$A_{e_i}(x,T)=0$$ for each $i=1,\dots,7$, and hence $$A_p(x,T)=0$$ for all $p\in \O$. This means $T$ is an $\O$-linear operator.
	\end{rem}

	\subsection{Octonionic isometric  operators on $\O$-Hilbert spaces}
	In this subsection, we discuss the $\O$-para-linear isometric operators on  $\O$-Hilbert spaces.

	We recall the notion of  orthonormal   bases on $H$. Due to  the non-associativity,  there are three kinds of orthonormal   bases.

	\begin{mydef}
		A subset $S=\{x_{\alpha}\}_{\alpha\in \Lambda}$ of $H$ is said to be an \textbf{orthonormal set} if $$\fx{x_{\alpha}}{x_{\beta}}=\delta_{\alpha \beta}.$$
		Further if $S\subseteq \re {H}$, then  $S$ is said to be an \textbf{associative orthonormal   set}.

		An orthonormal system $S=\{x_{\alpha}\}_{\alpha\in \Lambda}$ is said to be an \textbf{orthonormal basis} if there does not exist other orthonormal system $S'$ such that $S\subsetneqq S'$.
		Further if $S\subseteq \re {H}$, then  $S$ is said to be an \textbf{ associative orthonormal   basis}.
		An orthonormal basis (set) $S=\{x_{\alpha}\}_{\alpha\in \Lambda}$ is said to be a \textbf{weak associative orthonormal basis (set)} if  for all $\alpha,\beta\in \Lambda$,
		$$A_p(x_\alpha,x_\beta)=0.$$
	\end{mydef}

	It is easy to see that every orthonormal   basis of the real Hilbert space $(\re H,\left\langle \cdot,\cdot\right\rangle _{\spr})$ is automatically an associative orthonormal basis of $H$. 
	Due to the  non-associativity, the general  orthonormal bases of an octonionic Hilbert space may lead to some strange phenomena. For example,
	the 	cardinalities of  orthonormal bases of an $\O$-Hilbert space are not unique as shown in \cite{huoqinghai2021tensor}. This leads to the fact of the existence of octonionic Stiefel space $V_2(\O^4)$ \cite{qianchao2022}.
	Moreover, the Parseval theorem is not valid for general orthonormal bases.
	
	It turns out  that  the  notion of weak associative orthonormal basis is  the non-associative counterpart of orthonormal bases of classical case.  As shown in \cite{huoqinghai2021tensor},  the Parseval theorem holds for and only for   weak associative orthonormal base.
	Fortunately, in view of \cite[Lemma 4.8]{huoqinghai2021tensor}   the  cardinality of  weak associative orthonormal bases on an  $\O$-Hilbert space is constant.
	
	\begin{thm}[\textbf{\cite{huoqinghai2021tensor} Parseval theorem}]\label{thm:Parseval thm}
		Let $H$ be an $\O$-Hilbert space and $S=\{x_{\alpha}\}_{\alpha\in \Lambda}$ be a weak associative orthonormal   basis. Then  any $x\in H$  can be  uniquely expressed   as
		\begin{equation*}
			x=\sum_{\alpha\in \Lambda}\fx{x}{x_{\alpha}}x_{\alpha}
		\end{equation*}
		and there holds
		\begin{equation*}
			\fsh{x}^2=\sum_{\alpha\in \Lambda}|\fx{x}{x_{\alpha}}|^2.
		\end{equation*}
	\end{thm}

	
	\bigskip

	Let $H$ be a Hilbert $\O$-bimodule in this subsection. 
	Let $\Lambda $ be an index set such that its cardinality  is equal to that of  an arbitrary weak associative orthonormal basis of $H$. We have the following characterization of  octonionic isometric isomorphism.
	
	\begin{thm}\label{thm:equv isome}
		Let  $T:H\to H$ be a para-linear operator. Then the following are equivalent.
		\begin{enumerate}
			\item\label{it:T isom} $T$ is an  isometric isomorphism.
			\item\label{it:T map aonb to wonb} For any associative orthonormal basis $\{\epsilon_\alpha\}_{\alpha\in \Lambda}$ of $H$,
			$\{T\epsilon_\alpha\}_{\alpha\in \Lambda}$  is a weak associative orthonormal basis of $H$.
			\item  \label{it:T=Txe} For any associative orthonormal basis $\epsilon=\{\epsilon_\alpha\}_{\alpha\in \Lambda}$, there exists a unique weak associative  orthonormal basis $\xi=\{\xi_\alpha\}_{\alpha\in \Lambda}$ such that  
			$$Tx=\sum_{\alpha\in \Lambda} \fx{{x}}{{\xi_\alpha}}\epsilon_\alpha.$$
			\item \label{it:4 aonb imply wonb}There exists an associative orthonormal basis $\epsilon=\{\epsilon_\alpha\}_{\alpha\in \Lambda}$, such that 	$\{T\epsilon_\alpha\}_{\alpha\in \Lambda}$  is a weak associative orthonormal basis of $H$.
			
			\item\label{it:TT*=I}  $T^*\circ T=T\circ T^*=\text{Id}$. Here the notation ``$\circ$'' stands for the ordinary composition.
		\end{enumerate}
		Moreover, if $T$ is an $\O$-para-linear isometric isomorphism, then   
		\begin{eqnarray}\label{eq:T*T=I}
			T^*\circledcirc T=T\circledcirc T^*=\text{Id}.
		\end{eqnarray} Here the notation ``$\circledcirc$'' stands for the regular composition.
		
	\end{thm}
	\begin{proof}
		We prove \eqref{it:T isom} implies \eqref{it:T map aonb to wonb}.
		Let $\epsilon=\{\epsilon_\alpha\}_{\alpha\in \Lambda}$ be an arbitrary	associative orthonormal basis. We claim $\left(\{T\epsilon_\alpha\}_{\alpha\in \Lambda}\right)^{\perp}=0.$
		Since $T$ is surjective, it follows that for any $y\in H$, there exists $x\in H$ such that $y=Tx$. According to the Parseval theorem, we get $$x=\sum_{\alpha\in \Lambda}\fx{x}{\epsilon_\alpha}\epsilon_\alpha.$$
		Hence we conclude from Theorem \ref{cor: T re=0 iff T=0} that
		$$y=\sum_{\alpha\in \Lambda}T(\fx{x}{\epsilon_\alpha}\epsilon_\alpha)=\sum_{\alpha\in \Lambda}\fx{x}{\epsilon_\alpha}T(\epsilon_\alpha).$$
		Let $z\in \left(\{T\epsilon_\alpha\}_{\alpha\in \Lambda}\right)^{\perp}$. Then for all $y\in H$,
		$$\fx{z}{y}_\R=\re\fx{z}{\sum_{\alpha\in \Lambda}\fx{x}{\epsilon_\alpha}T(\epsilon_\alpha)}=\sum_{\alpha\in \Lambda}\re\left(\fx{z}{T(\epsilon_\alpha)}\overline{\fx{x}{\epsilon_\alpha}}\right)=0.$$
		This implies $z=0$ as desired.
		
		It follows from Lemma \ref{lem:isome} that
		$$\fx{T\epsilon_\alpha}{T\epsilon_\beta}=\fx{\epsilon_\alpha}{\epsilon_\beta}.$$
		Hence  $\{T\epsilon_\alpha\}_{\alpha\in \Lambda}$ is an orthonormal basis.
		Similarly, for any $\alpha\neq\beta$ and any $p\in \O$, it follows from Theorem \ref{cor: T re=0 iff T=0}  and Lemma \ref{lem:isome} that
		$$A_p(T\epsilon_\alpha,T\epsilon_\beta)=\fx{pT\epsilon_\alpha}{T\epsilon_\beta}-p\fx{T\epsilon_\alpha}{T\epsilon_\beta}=\fx{T(p\epsilon_\alpha)}{T\epsilon_\beta}=\fx{p\epsilon_\alpha}{\epsilon_\beta}=0.$$
		This proves that $\{T\epsilon_\alpha\}_{\alpha\in \Lambda}$ is a weak associative orthonormal basis.
		
		We next prove \eqref{it:T map aonb to wonb} implies \eqref{it:T=Txe}.
		We first show the uniqueness. If there is another basis $\{\xi'_\alpha\}_{\alpha\in\Lambda}$ such that 
		$$Tx=\sum_{\alpha\in \Lambda} \fx{{x}}{{\xi'_\alpha}}\epsilon_\alpha,$$
		then we have 
		$$\sum_{\alpha\in \Lambda} \fx{{x}}{{\xi'_\alpha}}\epsilon_\alpha=\sum_{\alpha\in \Lambda} \fx{{x}}{{\xi_\alpha}}\epsilon_\alpha$$
		for all $x\in H$.
		This implies that 
		$$\fx{{x}}{{\xi'_\alpha}}=\fx{{x}}{{\xi_\alpha}}$$
		for all $x\in H$ and $\alpha\in \Lambda$. Hence $\xi'_\alpha=\xi_\alpha$ for all $\alpha\in \Lambda$.
		
		By  assertion   \eqref{it:T map aonb to wonb}, it follows that $\{T\epsilon_\beta\}_{\beta\in \Lambda}$ is a  weak associative  orthonormal basis. In view of the Parseval theorem  , we conclude that
		$$\sum_{\beta\in \Lambda}\abs{\fx{\epsilon_{\alpha}}{T\epsilon_{\beta}}}^2=\fsh{\epsilon_{\alpha}}^2=1.$$
		Hence the series $$\sum_{\beta\in \Lambda}\fx{\epsilon_{\alpha}}{T\epsilon_{\beta}}\epsilon_{\beta}$$ converses absolutely.
		We define 	$$\xi_\alpha=\sum_{\beta\in \Lambda}\fx{\epsilon_{\alpha}}{T\epsilon_{\beta}}\epsilon_{\beta}.$$
		Note that the second associator $A_p(u,v)$ of $H$ vanishes if $u$ or $v$ is an associative element.
		It follows from identities \eqref{eq:<u,pv>=<u,v>p^+Ap(uv)} and \eqref{eq:<pu,pv>in lemma} that
		\begin{eqnarray*}
			\fx{\xi_\alpha}{\xi_\gamma}&=&\fx{\sum_{\beta\in \Lambda}\fx{\epsilon_{\alpha}}{T\epsilon_{\beta}}\epsilon_{\beta}}{\sum_{\beta\in \Lambda}\fx{\epsilon_{\gamma}}{T\epsilon_{\beta}}\epsilon_{\beta}}\\
			&=&\sum_{\beta\in \Lambda}\fx{\epsilon_{\alpha}}{T\epsilon_{\beta}}\overline{\fx{\epsilon_{\gamma}}{T\epsilon_{\beta}}}\\	&=&\sum_{\beta\in \Lambda}\fx{\epsilon_{\alpha}}{\fx{\epsilon_{\gamma}}{T\epsilon_{\beta}}T\epsilon_{\beta}}-A_{\fx{\epsilon_{\gamma}}{T\epsilon_{\beta}}}(\epsilon_{\alpha},T\epsilon_{\beta})\\
			&=&\sum_{\beta\in \Lambda}\fx{\epsilon_{\alpha}}{\fx{\epsilon_{\gamma}}{T\epsilon_{\beta}}T\epsilon_{\beta}}.
		\end{eqnarray*}
		Note again that $\{T\epsilon_\beta\}_{\beta\in \Lambda}$ is a  weak associative  orthonormal basis by assertion  \eqref{it:T map aonb to wonb}.  We thus conclude from the Parseval theorem  \ref{thm:Parseval thm} that
		$$	\fx{\xi_\alpha}{\xi_\gamma}=\fx{\epsilon_\alpha}{\epsilon_\gamma}=\delta_{\alpha \gamma}.$$
		This proves that  $\xi=\{\xi_\alpha\}_{\alpha\in \Lambda}$ is an orthonormal system.
		For any $z\in( \{\xi_\alpha\}_{\alpha\in \Lambda})^\perp$, it follows from Lemma \ref{lem:the second associator of H 's prop} and the fact that  $\epsilon_\beta$ is associative element that
		\begin{eqnarray*}
			0&=&\fx{z}{\xi_\alpha}\\
			&=&\fx{z}{\sum_{\beta\in \Lambda}\fx{\epsilon_\alpha}{T\epsilon_\beta}\epsilon_\beta}\\
			&=&\sum_{\beta\in \Lambda}\fx{z}{\epsilon_\beta}\overline{\fx{\epsilon_\alpha}{T\epsilon_\beta}}\\
			&=&\sum_{\beta\in \Lambda}\fx{z}{\epsilon_\beta}\fx{T\epsilon_\beta}{\epsilon_\alpha}\\
			&=&\fx{\sum_{\beta\in \Lambda}\fx{z}{\epsilon_\beta}T\epsilon_\beta}{\epsilon_\alpha}.
		\end{eqnarray*}
		Since $\{\epsilon_{\alpha}\}_{\alpha\in \Lambda}$ is an associative orthonormal basis. It follows  that 
		$$\sum_{\beta\in \Lambda}\fx{z}{\epsilon_\beta}T\epsilon_\beta=0.$$
		Note that $\{T\epsilon_{\alpha}\}_{\alpha\in \Lambda}$ is a weak associative orthonormal basis.
		This implies that $$\fx{z}{\epsilon_\beta}=0$$ for all $\beta\in \Lambda$. Thus we get that $z=0$.
		This shows that $\xi=\{\xi_\alpha\}_{\alpha\in \Lambda}$ is an orthonormal basis.
		Similarly, for any $p\in \O$ and $\alpha\neq \gamma$, we have
		\begin{eqnarray*}
			A_p({\xi_\alpha},{\xi_\gamma})
			&=&\fx{\sum_{\beta\in \Lambda}\fx{p\epsilon_{\alpha}}{T\epsilon_{\beta}}\epsilon_{\beta}}{\sum_{\beta\in \Lambda}\fx{\epsilon_{\gamma}}{T\epsilon_{\beta}}\epsilon_{\beta}}\\
			&=&\sum_{\beta\in \Lambda}\fx{p\epsilon_{\alpha}}{T\epsilon_{\beta}}\overline{\fx{\epsilon_{\gamma}}{T\epsilon_{\beta}}}\\
			&=&\sum_{\beta\in \Lambda}\fx{p\epsilon_{\alpha}}{T\epsilon_{\beta}}\fx{T\epsilon_{\beta}}{\epsilon_{\gamma}}\\
			&=&	\sum_{\beta\in \Lambda}\fx{\fx{p\epsilon_{\alpha}}{T\epsilon_{\beta}}T\epsilon_{\beta}}{\epsilon_{\gamma}}\\
			&=&\fx{p\epsilon_{\alpha}}{\epsilon_{\gamma}}\\
			&=&0.
		\end{eqnarray*}
		This proves that  $\xi=\{\xi_\alpha\}_{\alpha\in \Lambda}$ is a weak associative orthonormal basis.
		
		Define 
		\begin{eqnarray}
			S:H&\to& H\notag\\
			y&\mapsto &\sum_{\alpha\in \Lambda} \fx{{y}}{{\xi_\alpha}}\epsilon_\alpha.\label{eqdef:Txiep}
		\end{eqnarray}
		It is easy to check that $S$ is a para-linear operator. To prove that $T=S$, in view of  Theorem \ref{cor: T re=0 iff T=0},    we only need to check the equality for an associative orthonormal basis.
		For any $\gamma\in \Lambda$, we have
		\begin{eqnarray*}
			S(\epsilon_\gamma)&=&\sum_{\alpha\in \Lambda} \fx{\epsilon_\gamma}{{\sum_{\beta\in \Lambda}\fx{\epsilon_{\alpha}}{T\epsilon_{\beta}}\epsilon_{\beta}}}\epsilon_\alpha\\
			&=&\sum_{\alpha\in \Lambda} \overline{{\fx{\epsilon_{\alpha}}{T\epsilon_{\gamma}}}}\epsilon_\alpha\\
			&=&\sum_{\alpha\in \Lambda} {\fx{T\epsilon_{\gamma}}{\epsilon_{\alpha}}}\epsilon_\alpha\\
			&=&T\epsilon_\gamma.
		\end{eqnarray*}
		This proves assertion \eqref{it:T=Txe}.

		Next, we prove that  \eqref{it:T=Txe} implies \eqref{it:T isom}.
		Choose an arbitrary  associative orthonormal basis $\epsilon=\{\epsilon_\alpha\}_{\alpha\in \Lambda}$.  By assertion \eqref{it:T=Txe}, there exists a unique weak associative  orthonormal basis $\xi=\{\xi_\alpha\}_{\alpha\in \Lambda}$ such that
		$Tx=\sum_{\alpha\in \Lambda} \fx{{x}}{{\xi_\alpha}}\epsilon_\alpha$.
		For any $x\in H$, by direct calculations  we have
		\begin{eqnarray*}
			\fsh{Tx}^2
			&=&\fx{\sum_{\alpha\in \Lambda} \fx{{x}}{{\xi_\alpha}}\epsilon_\alpha}{\sum_{\beta\in \Lambda} \fx{{x}}{{\xi_\beta}}\epsilon_\beta}\\
			&=&\sum_{\alpha\in \Lambda} \abs{\fx{{x}}{{\xi_\alpha}}}^2.
		\end{eqnarray*}
		Note that $\xi=\{\xi_\alpha\}_{\alpha\in \Lambda}$ is a weak associative  orthonormal basis. Hence it follows from the Parseval theorem  \ref{thm:Parseval thm} that
		$$\fsh{Tx}=\fsh{x}.$$
		Note that an isometry is always injective. We show that  $T$ is surjective.
		For any $y\in H$, set $$x=\sum_{\alpha\in \Lambda} \fx{y}{\epsilon_\alpha}\xi_\alpha.$$
		Then $$Tx=\sum_{\alpha\in \Lambda}\fx{x}{\xi_\alpha}\epsilon_\alpha=\sum_{\alpha\in \Lambda}\fx{y}{\epsilon_\alpha}\epsilon_\alpha=y.$$
		This proves that $T$ is surjective.
		
		Now we have shown that  \eqref{it:T isom}, \eqref{it:T map aonb to wonb}
		and \eqref{it:T=Txe}  are equivalent to each other. It is easy to see that \eqref{it:T map aonb to wonb} implies \eqref{it:4 aonb imply wonb}. We show that \eqref{it:4 aonb imply wonb} implies \eqref{it:T isom}. Let $\epsilon=\{\epsilon_\alpha\}_{\alpha\in \Lambda}$ be an 	associative orthonormal basis such that $\{T\epsilon_\alpha\}_{\alpha\in \Lambda}$  is a weak associative orthonormal basis of $H$.
		Since $T$ is para-linear, for any $x=\sum_{\alpha\in \Lambda}\fx{x}{\epsilon_{\alpha}}\epsilon_{\alpha}$, it follows that 
		$$Tx=\sum_{\alpha\in \Lambda}\fx{x}{\epsilon_{\alpha}}T\epsilon_{\alpha}.$$
		Note that $\{T\epsilon_\alpha\}_{\alpha\in \Lambda}$  is a weak associative orthonormal basis of $H$.  It is easy to prove that $T$ is surjective. We conclude from the Parseval theorem \ref{thm:Parseval thm} that
		$$\fsh{Tx}^2=\sum_{\alpha\in \Lambda}\abs{\fx{x}{\epsilon_{\alpha}}}^2=\fsh{x}^2$$ for all $x\in H$. This shows \eqref{it:T isom}.

		It is a classical result that  \eqref{it:TT*=I} is equivalent to \eqref{it:T isom}.
		
		We finally prove that \eqref{it:TT*=I} implies \eqref{eq:T*T=I}. Note that $T^*$ is para-linear. In view of Theorem \ref{cor: T re=0 iff T=0}, to obtain \eqref{eq:T*T=I}, it is sufficient to show that 
		$$	(T^*\circledcirc T)(x)=(T\circledcirc T^*)(x)=\text{Id}x$$
		for any $x\in \re H$. By Lemma \ref{lem:vanishing[f,g,x]=0}, this is equivalent to showing  
		$$	(T^*\circ T)(x)=(T\circ T^*)(x)=x$$ for any $x\in \re H$, which holds obviously by \eqref{it:TT*=I}. 
		This completes the proof.
	\end{proof}
	
	\begin{rem}
		
		This theorem shows that each para-linear isometric isomorphism corresponds to a unique weak associative orthonormal basis up to the choice of an associative orthonormal basis. Conversely, for each weak associative orthonormal basis $\{\xi_\alpha\}_{\alpha\in \Lambda}$, we can obtain two  para-linear isometric isomorphisms defined by 
		$$T_1x=\sum_{\alpha\in \Lambda} \fx{{x}}{{\xi_\alpha}}\epsilon_\alpha, \qquad T_2x=\sum_{\alpha\in \Lambda} \fx{{x}}{{\epsilon_\alpha}}\xi_\alpha.$$ 
		Here  $\{\epsilon_\alpha\}_{\alpha\in \Lambda}$ is an arbitrary associative orthonormal basis.
	\end{rem}
	
	We next discuss partial isometric isomorphism.
	Let $H$ be a Hilbert $\O$-bimodule and $T$ be a para-linear operator on $H$. $T$ is called a \textbf{partial isometry} if $T|_{(\ker T)^\perp}$ is an isometry. 
	
	\begin{thm}\label{thm:p.iso}
		Let $T$ be a para-linear operator on $H$. Suppose $\ker T$ is an $\O$-submodule of $H$.
		Then the following are equivalent.
		\begin{enumerate}
			\item \label{it: T is p.iso}$T$ is a partial isometry.
			\item \label{it: T AONS =WONS} For any 	associative orthonormal  set $\{\epsilon_\alpha\}_{\alpha\in \Lambda}$ in $(\ker T)^\perp$, $\{T\epsilon_\alpha\}_{\alpha\in \Lambda}$ is a weak associative orthonormal  set of $H$.
			\item \label{it:'T AONS=WONS}There exists an 	associative orthonormal  basis $\{\epsilon_\alpha\}_{\alpha\in \Lambda}$ of Hilbert $\O$-submodule $(\ker T)^\perp$, such that $\{T\epsilon_\alpha\}_{\alpha\in \Lambda}$ is a weak associative orthonormal  set of $H$.
			
		\end{enumerate}
		
	\end{thm}
	\begin{proof}
		We prove  that \eqref{it: T is p.iso} implies \eqref{it: T AONS =WONS}.
		Let $\{\epsilon_\alpha\}_{\alpha\in \Lambda}$ be an arbitrary  	associative orthonormal  set  in $(\ker T)^\perp$.
		Recall the  polarization identity (\cite[formula (4.17)]{huoqinghai2022Riesz}):
		$$\fx{x}{y}=\sum_{i=0}^7e_i(\fsh{e_ix+y}^2-\fsh{e_ix-y}^2).$$
		Therefore it follows from assertion \eqref{it: T is p.iso} that 
		\begin{eqnarray*}
			\fx{T\epsilon_\alpha}{T\epsilon_\beta}&=&\sum_{i=0}^7e_i(\fsh{e_iT\epsilon_\alpha+T\epsilon_\beta}^2-\fsh{e_iT\epsilon_\alpha-T\epsilon_\beta}^2)\\
			&=&
			\sum_{i=0}^7e_i(\fsh{T(e_i\epsilon_\alpha+\epsilon_\beta)}^2-\fsh{T(e_i\epsilon_\alpha-\epsilon_\beta)}^2)\\
			&=&	\sum_{i=0}^7e_i(\fsh{e_i\epsilon_\alpha+\epsilon_\beta}^2-\fsh{e_i\epsilon_\alpha-\epsilon_\beta}^2)\\&=&\fx{\epsilon_\alpha}{\epsilon_\beta}.
		\end{eqnarray*}  
		Hence $\{T\epsilon_\alpha\}_{\alpha\in \Lambda}$ is an  orthonormal set of $H$. For any $p\in \O$, replacing $\epsilon_{\beta}$ with $p\epsilon_{\beta}$ above, we obtain 
		$$\fx{T\epsilon_\alpha}{Tp\epsilon_\beta}=\fx{\epsilon_\alpha}{p\epsilon_\beta}.$$
		Thus for any $\alpha\neq \beta$ in $\Lambda$, we have $$A_p(T\epsilon_\beta,T\epsilon_\alpha)=\fx{pT\epsilon_\beta}{T\epsilon_\alpha}=\overline{\fx{T\epsilon_\alpha}{Tp\epsilon_\beta}}=\overline{\fx{\epsilon_\alpha}{p\epsilon_\beta}}=0.$$
		This shows that $\{T\epsilon_\alpha\}_{\alpha\in \Lambda}$ is a weak  associative orthonormal set of $H$.
		
		It is obvious that \eqref{it: T AONS =WONS} implies \eqref{it:'T AONS=WONS}. We prove that  \eqref{it:'T AONS=WONS} implies \eqref{it: T is p.iso}.
		Note that $\ker T$ is a Hilbert $\O$-submodule of $H$.
		We can extend the associative orthonormal  basis $\{\epsilon_\alpha\}_{\alpha\in \Lambda}$ of Hilbert $\O$-submodule $(\ker T)^\perp$ to an associative orthonormal  basis $\{\epsilon_\alpha\}_{\alpha\in \Lambda}\cup \{\epsilon_{\alpha'}\}_{{\alpha'}\in \Lambda'}$ of $H$, where $\{\epsilon_{\alpha'}\}_{{\alpha'}\in \Lambda'}$ is an associative orthonormal  basis of $\ker T$. 
		For any $x=\sum_{\alpha\in \Lambda}\fx{x}{\epsilon_\alpha}\epsilon_\alpha+\sum_{{\alpha'}\in \Lambda'}\fx{x}{\epsilon_{\alpha'}}\epsilon_{\alpha'}$, it follows from the para-linearity of $T$ that
		$$Tx=\sum_{\alpha\in \Lambda}\fx{x}{\epsilon_\alpha}T\epsilon_\alpha.$$
		In view of the proof of Parseval theorem, for any given $x\in H$, the sum above is in fact countable (only countable summands are not zero)
		$$Tx=\sum_{i=1}^\infty\fx{x}{\epsilon_{\alpha_i}}T\epsilon_{\alpha_i}.$$
		To  prove that $T|_{(\ker T)^\perp}$ is an isometry, we recall an identity (\cite[formula (34)]{huoqinghai2021tensor}):
		\begin{eqnarray}\label{eq:bessel}
			\fsh{y}^2=\sum_{i=1}^N\abs{\fx{y}{\xi_i}}^2+\fsh{y-\sum_{i=1}^N\fx{y}{\xi_i}\xi_i}^2,
		\end{eqnarray}
		where $\{\xi_i\}_{i=1}^N$ is a weak associative orthonormal set.
		
		For any $x\in (\ker T)^\perp$, denote  $$y_N=\sum_{i=1}^N\fx{x}{\epsilon_{\alpha_i}}T\epsilon_{\alpha_i}.$$  Notice that $\{T\epsilon_{\alpha_i}\}_{i=1}^\infty$ is a weak associative orthonormal set.  It follows that 
		$$\fx{y_N}{T\epsilon_{\alpha_i}}=\fx{x}{\epsilon_{\alpha_i}}.$$
		Therefore we conclude from the identity \eqref{eq:bessel} that
		\begin{eqnarray*}
			\fsh{Tx}^2&=&\fsh{\lim_{N\to \infty}\sum_{i=1}^N\fx{x}{\epsilon_{\alpha_i}}T\epsilon_{\alpha_i}}^2\\
			&=&\lim_{N\to \infty}\fsh{y_N}^2\\
			&=&\lim_{N\to \infty}\left(\sum_{i=1}^N\abs{\fx{y_N}{T\epsilon_{\alpha_i}}}^2+	\fsh{y_N-\sum_{i=1}^N\fx{y_N}{T\epsilon_{\alpha_i}}T\epsilon_{\alpha_i}}^2\right)\\
			&=&\lim_{N\to \infty}\sum_{i=1}^N\abs{\fx{x}{\epsilon_{\alpha_i}}}^2.
		\end{eqnarray*}
		Note that $\{\epsilon_\alpha\}_{\alpha\in \Lambda}$ is an  associative orthonormal  basis  of the  Hilbert $\O$-submodule $(\ker T)^\perp$. It follows that 
		$$\fsh{x}^2=\sum_{\alpha\in \Lambda}\abs{\fx{x}{\epsilon_{\alpha_i}}}^2=\lim_{N\to \infty}\sum_{i=1}^N\abs{\fx{x}{\epsilon_{\alpha_i}}}^2.$$
		Thus $$\fsh{Tx}=\fsh{x}$$ for all $x\in (\ker T)^\perp$.
		This proves that $T$ is a partial isometry.\end{proof}
	
	\begin{rem}
		
		This theorem shows that each para-linear partial isometry corresponds to a unique weak associative orthonormal set up to a choice of associative orthonormal set. Conversely, given an arbitrary associative orthonormal set $\{\epsilon_\alpha\}_{\alpha\in \Lambda}$, for each weak associative orthonormal set $\{\xi_\alpha\}_{\alpha\in \Lambda}$, we can obtain a  para-linear partial isometry defined by 
		$$Tx=	\sum_{\alpha\in \Lambda} \fx{{x}}{{\epsilon_\alpha}}\xi_\alpha.$$
		Note that if $$	\sum_{\alpha\in \Lambda} \fx{{x}}{{\epsilon_\alpha}}\xi_\alpha=0,$$
		then for each $\alpha\in \Lambda$, $$\fx{{x}}{{\epsilon_\alpha}}=\fx{	\sum_{\beta\in \Lambda} \fx{{x}}{{\epsilon_\beta}}\xi_\beta}{\xi_\alpha}=0.$$ This shows  that 
		$$(\ker T)^\perp=\left\langle \{\epsilon_{\alpha}\}_{\alpha\in \Lambda}\right\rangle_{\O}.$$
		The rest proof of $T$ being a para-linear partial isometry runs in the same way as in the proof of Theorem \ref{thm:p.iso}.
	\end{rem}

	As a byproduct, we can prove that the set of weak associative orthonormal bases is closed under the unit octonionic scalar multiplication.
	To this end, we introduce the octonionic module structure of $\Hom_{\mathcal{LO}}(H,H)$ on a Hilbert $\O$-bimodule. 
	Let $T$ be a para-linear operator on $H$. Define 
{	\begin{eqnarray}
			(T\odot  p)(x):=T(x)p-A_p(x,T)\label{eq:f p}\\
			(p\odot  T)(x):=T(xp)+A_p(x,T)\label{eq:p f}
	\end{eqnarray}  }
	for any $x\in H$ and $p\in \O$.
	It is shown that $\Hom_{\mathcal{LO}}(H,H)$ is an $\O$-bimodule with respect to the octonionic scalar multiplication \eqref{eq:f p} and \eqref{eq:p f} (see \cite{huoqinghai2020nonass}).
	
	Similar to \cite[Lemma 5.1]{huoqinghai2022Riesz}, we can generalize the Moufang identity as follows.
	\begin{lemma}
		
		Let $T$ be a para-linear operator on $H$. Then for any octonion $p\neq 0$,
		we have
		{	\begin{eqnarray}
				(T\odot  p)(x)=pT(p^{-1}x)p\label{eq:T p}\\
				(p\odot  T)(x)=p^{-1}T(pxp)\label{eq:p T}
		\end{eqnarray}}
		for all $x\in H$.
	\end{lemma}
	\begin{proof}
		\bfs \ $\abs{p}=1$. Recall  identity (\cite[Proposition 3.6]{huoqinghai2020nonass}) $$A_p(\overline{p}x,T)=pA_p(x,T) $$ and identity  (3.32) in \cite[Proposition 3.15]{huoqinghai2020nonass}:
		$$pA_p(x,T)=A_p(x,T)\overline{p}.$$ 
		It follows that 
		\begin{eqnarray*}
			(T\odot  p)(x)p^{-1}&=&((Tx)p-A_p(x,T))\overline{p}\\
			&=&(Tx)-pA_p(x,T)\\
			&=&(Tx)-A_p(\overline{p}x,T)\\
			&=&pT(p^{-1}x).
		\end{eqnarray*}
		This implies \eqref{eq:T p}.

		By definition, we have
		$$	(p\odot  T)(x)=(p\odot  T)((xp)\overline{p})=T(xp)+A_p((xp)\overline{p},T).$$
		Recall identity ((3.39) in \cite[Proposition 3.19]{huoqinghai2020nonass}):
		$$A_p(xp,T)=pA_p(x,T).$$
		We conclude that
		\begin{eqnarray*}
			(p\odot  T)(x)		&=&T(xp)+\overline{p}A_p(xp,T)\\
			&=&T(xp)+\overline{p}T(pxp)-\overline{p}(pT(xp))\\
			&=&p^{-1}T(pxp).
		\end{eqnarray*}
		This proves \eqref{eq:p T}.
	\end{proof}
	\begin{thm}\label{thm:pxi is wonb}
		Let $\{\xi_\alpha\}_{\alpha\in \Lambda}$ be an arbitrary weak associative orthonormal basis (set). Then for any $p\in \O$ with $\abs{p}=1$,  
		$\{p\xi_\alpha\}_{\alpha\in \Lambda}$ and $\{\xi_\alpha p \}_{\alpha\in \Lambda}$ are also weak associative orthonormal bases (sets).
	\end{thm}
	\begin{proof}
		For any given associative orthonormal basis (set) $\{\epsilon_\alpha  \}_{\alpha\in \Lambda}$, we define
		$$Tx:=\sum_{\alpha\in \Lambda}\fx{x}{\epsilon_{\alpha}}\xi_\alpha.$$
		It is shown that $T$ is an isometric isomorphism (partial isometry).
		For any $p\in \O$ with $\abs{p}=1$, $p\odot T$ and $T\odot p$ are para-linear operators. In view of \eqref{eq:p T} and \eqref{eq:T p}, we obtain if $T$ is surjective, then so are $p\odot T$ and $T\odot p$, and
		$$\fsh{(T\odot p)(x)}=\fsh{pT(p^{-1}x)p}=\fsh{T(p^{-1}x)}=\fsh{p^{-1}x}=\fsh{x}$$ and 
		$$\fsh{(p\odot  T)(x)}=\fsh{p^{-1}T(pxp)}=\fsh{T(pxp)}=\fsh{pxp}=\fsh{x}$$
		for any $x\in H$ ($x\in (\ker T)^\perp$). This shows that  $p\odot T$ and $T\odot p$ are para-linear isometric isomorphisms (partial isometries). By Theorem \ref{thm:equv isome} (Theorem \ref{thm:p.iso}), we get
		$\{(T\odot p)(\epsilon_{\alpha})\}_{\alpha\in \Lambda}$, $\{(p\odot T)(\epsilon_{\alpha})\}_{\alpha\in \Lambda}$ are  weak associative orthonormal bases (sets).
		Since $\epsilon_{\alpha}\in \re H$, it follows that
		$$(T\odot p)(\epsilon_{\alpha})=T(\epsilon_{\alpha})p=\xi_\alpha p, \qquad (p\odot T)(\epsilon_{\alpha})=T(\epsilon_{\alpha}p)=T(p\epsilon_{\alpha})=p\xi_\alpha.$$
		This shows that $\{p\xi_\alpha\}_{\alpha\in \Lambda}$ and $\{\xi_\alpha p \}_{\alpha\in \Lambda}$ are also weak associative orthonormal bases (sets).
	\end{proof}
	\subsection{$\O$-isometric transformations on $\O^n$}
	
	We consider  the $\O$-isometric transformations on $\O^n$ in this section.
	It turns out that $\O$-isometric transformations on $\O^n$
	can be regarded as    matrices on $\O^n$ whose row vectors or conjugate transpose of  column vectors form a weak associative basis.
{The notation $\O^n$ shall  represents the set of all row vectors $(x_1,\dots,x_n)$, where $x_i\in \O$ for $i=1,\dots,n$.
Throughout, we use bold symbols  to  represent vectors of $\O^n$ or matrix on $\O^n$.	The octonionic inner product on $\O^n$ is 
	$$\fx{\pmb x}{\pmb y}=\sum_{i=1}^nx_i\overline{y_i}$$ for all row vectors  $\pmb x=(x_1,\dots,x_n),\pmb y=(y_1,\dots,y_n)\in \O^n$. 
	 }

	We denote by $$\pmb {\epsilon^i}=(0,\ldots,0,1,0,\ldots,0)$$ the row vector  with $1$ in the $i$-th slot. Then $$\pmb{\epsilon}:=\{\pmb {\epsilon^i}\}_{i=1}^n$$ is an associative orthonormal basis of  $\O^n$. 	
Let  $T$ be an octonionic para-linear  operator on $\O^n$.  
	There is an octonionic   matrix
	corresponding to $T$ with respect  to the associative orthonormal basis $\pmb{\epsilon}$. Concretely,
	denote  $$\pmb{x^i}:=T(\pmb{\epsilon^i})$$ for each $i=1,\dots,n$.
	Then for  any $\pmb{y}=(y_1,\dots,y_n)\in \O^n$,
	\begin{eqnarray}\label{eq:Ty=yT}
		T(\pmb{y})=(y_1,\dots,y_n)_{1\times n} \
		\left(\begin{array}[c]{ccc}
			\pmb{x^1}\\
			\vdots\\
			\pmb{x^n}
		\end{array}\right)_{n\times n}.	
	\end{eqnarray}
	Note that here we use the right action of a matrix since $T$ is left para-linear.
	Hence we can think of $T$ as  the matrix
	$$\pmb{T}:=\left(\begin{array}[c]{ccc}
		\pmb{x^1}\\
		\vdots\\
		\pmb{x^n}
	\end{array}\right)_{n\times n}$$
	so that
	$$T(\pmb{y})=\pmb{y}_{1\times n}\pmb{T}_{n\times n},$$
	where the multiplication of the right hand side is the octonionic  matrix multiplication.
	Suppose $$\pmb{x^i}=(x{^i}_1,\dots,x^i{_n}),$$
	for each $i=1,\dots,n$. {Let $\pmb{x_i}^{\dag}$ denote    the conjugate transpose of the  column vector of the matrix  $\pmb{T}$
	$$\pmb{x_i}^{\dag}:=\left(\begin{array}[c]{ccc}
		\overline{{x^1}_{i}}\\
		\vdots\\
		\overline{{x^n}_{i}}
	\end{array}\right)_{n\times 1}^T=(\overline{{x^1}_{i}},\dots,\overline{{x^n}_{i}}).$$
Here, for a column vector  $\pmb{a}$, we denote its transpose by  $\pmb{a}^T$. With this notation, we write
	\begin{eqnarray*}
		T  (	\pmb{y})&=&\sum_{i=1}^n \fx{\pmb{y}}{\pmb{x_i}^{\dag}}\pmb {\epsilon^i}.
	\end{eqnarray*}
}

	Since the dual operator $T^*$ of a para-linear operator $T$ is also para-linear by Lemma \ref{lem:T* is paralin}, it is easy to show that the  corresponding matrix of $T^*$ is the conjugate transpose of the matrix of $T$.
	Note that each isometric transformation on $\O^n$ is automatically a bijection. We can translate Theorem \ref{thm:equv isome} to the finite dimensional case as follows.
	
	\begin{thm}\label{thm:ft equv}
		Let $T$ be a left para-linear operator on $\O^n$ and  $\pmb{T}$ be the matrix of $T$ in the sense of \eqref{eq:Ty=yT}. Then
		the following are equivalent.
		\begin{enumerate}
			\item \label{it:isom}$T$ is  an $\O$-isometry.
			\item \label{it:map basis to basis} $T$ maps any associative orthonormal basis to a weak associative orthonormal basis.
			\item \label{it:row is wonb} The row vectors $\{\pmb{x^i}\}_{i=1}^n$  of $\pmb{T}$ form a weak  associative orthonormal basis of $\O^n$.
			\item \label{it:column is wonb} {The conjugate transpose of the  column vectors
			$\{\pmb{x_i}^\dag\}_{i=1}^n$ of the matrix  $\pmb{T}$  form a weak  associative orthonormal basis of $\O^n$.}
			
			\item \label{it:TT*=T*T=I}
			$T^*\circ T=T\circ T^*=\text{Id}$.
			
		\end{enumerate}
		Let $\pmb{T}$ be an octonionic matrices  and $\pmb{ T^*} $ denote its dual operator. Then we have:
		$$\pmb{T}\pmb{ T^*}=\pmb{\text{Id}}=\pmb{ T^*}\pmb{T}.$$
	\end{thm}

	\begin{rem}
		
		It is worthy noticing that  $\pmb{T}\pmb{ T^*}=\pmb{\text{Id}}$ is not equivalent to $\pmb{T^*}\pmb{T}=\pmb{\text{Id}}$ in the octonionic setting \cite{qianchao2022}. Let $T$ be the operator whose matrix is $$\dfrac{1}{\sqrt{2}}\left(\begin{array}[c]{ccc}
			e_7& e_3\\
			e_2& -e_6
		\end{array}\right).$$
		It follows that 
		$$\pmb{T}\pmb{ T^*}=\pmb{\text{Id}},\quad \pmb{T^*}\pmb{T}=\left(\begin{array}[c]{ccc}
			1& -e_4\\
			e_4& 1
		\end{array}\right).$$
		It can be seen that its row vectors  form an orthonormal basis while its column vectors are not.
		
		Moreover, we find that there exists an octonionic matrix $\pmb{T}$ satisfying $$\pmb{T}\pmb{ T^*}=\pmb{\text{Id}}=\pmb{ T^*}\pmb{T}$$ but not being  an isometry.
		Let $$\pmb{T}=\dfrac{1}{\sqrt{2}}\left(\begin{array}[c]{ccc}
			1&e_1 \\
			-e_3& e_2
		\end{array}\right).$$
		It can be seen that its row vectors  form an orthonormal basis  of $\mathbb H_{e_1,e_2}^2$. Hence we have$$\pmb{T}\pmb{ T^*}=\pmb{\text{Id}}=\pmb{ T^*}\pmb{T}.$$
		However, $(1,e_1 ),
		(	-e_3, e_2)$ is not  a weak associative orthonormal basis.  In fact, for any $p\in \O$,
		$$A_p((1,e_1 ),
		(	-e_3, e_2))=pe_3+(pe_1)\overline{e_2}=-[p,e_1,e_2].$$
		It is clear that there exists $p\in \O$ such that $A_p((1,e_1 ),
		(	-e_3, e_2))\neq0$.
	\end{rem}

	Similarly, we translate Theorem \ref{thm:p.iso} to the finite dimensional case as follows.
	\begin{thm}
		
		For any para-linear partial isometry on $\O^n$, there exists an associative orthonormal basis with respect which, the matrix of $T$ is of the form
		$$\left(\begin{array}[c]{ccc}
			\pmb{x^1}\\
			\vdots\\
			\pmb{x^k}\\
			\pmb{0}\\
			\vdots\\
			\pmb{0}
		\end{array}\right)_{n\times n}.$$
		Here $(\pmb{x^1},\dots, 	\pmb{x^k})$ is a weak  associative orthonormal set of $\O^n$.
		
		Conversely, for any   weak  associative orthonormal set $(\pmb{x^1},\dots, 	\pmb{x^k})$ of $\O^n$, the matrix $$\left(\begin{array}[c]{ccc}
			\pmb{x^1}\\
			\vdots\\
			\pmb{x^k}\\
			\pmb{0}\\
			\vdots\\
			\pmb{0}
		\end{array}\right)_{n\times n}$$
		defines a para-linear partial isometry on $\O^n$.
	\end{thm}
	
	We denote the set of all para-linear isometric matrices by
	$$  \text{Iso}_\spf(n):=\{T:\spf^n\to \spf^n\mid T \text{ is an $\spf$-para-linear isometry.}\}$$ for $\spf=\R,\spc,\mathbb H,\O$. Note that the para-linearity is just the canonical linearity in associative division algebras. It is clear that $  \text{Iso}_\spf(n)$ is a closed subset of the matrices space $M_\spf(n)$.
	It is worth pointing out that the octonionic para-linearity is not closed under the ordinary composition, while property \eqref{it:T map aonb to wonb} in Theorem \ref{thm:ft equv} is not closed under the ordinary octonionic   matrix multiplication. Hence $\text{Iso}_\spf(n)$ is not a group  under the ordinary composition or regular composition.

	\begin{eg}
		We give some examples of para-linear isometries.
		\begin{enumerate}
			\item $\text{Iso}_\spo(1)\cong \S^7$.
			
			\medskip
			
			\item

			$
			\dfrac{1}{\sqrt{2}}\left(\begin{array}[c]{ccc}
				e_1& e_2\\
				e_2& e_1
			\end{array}\right),\quad
			\dfrac{1}{\sqrt{2}}\left(\begin{array}[c]{ccc}
				e_1& e_2\\
				e_1& -e_2
			\end{array}\right)\in \text{Iso}_\spo(2).$
			
			\medskip
			
			\item $
			\dfrac{1}{\sqrt{6}}\left(\begin{array}[c]{ccc}
				{\sqrt{2}e_4}&	\sqrt{2}e_1&\sqrt{2}e_2\\
				0&	\sqrt{3}e_2& \sqrt{3}e_1\\
				-2e_4&	e_1& e_2	
			\end{array}\right)\in \text{Iso}_\spo(3).$
		\end{enumerate}
		
	\end{eg}
	\begin{thm}\label{thm:iso(2)}
		$\text{Iso}_\spo(2)\cong\bigcup_{J\in \S^6}\S^7\times U_{\spc_J}(2)/ \sim $, where $\sim$ is an equivalence relation defined by
		$$(p,\pmb{T})\sim(q,\pmb{S})\iff p\pmb{T}=q\pmb{S}.$$
		Here $\S^6$ is the set of imaginary units of $\O$  and
		$U_{\spc_J}(2):=U({\spc_J}^{2})$ is the unitary group of ${\spc_J}^{2}$.
		
	\end{thm}
	\begin{proof}
		By Theorem \ref{thm:ft equv}, we know that
		$$\text{Iso}_\spo(2)=\left\{\left(
		\begin{array}[c]{ccc}
			a& b\\
			c& d
		\end{array}\right): ((a,b), (c,d)) \text{ is a weak associative orthonormal basis of }  \O^2 \right\}.$$
		Suppose $a\in \R$.	Since $A_p((a,b), (c,d))=0$ for all $p\in \O$, it follows from $a\in \R$ that 
		$$[p,c,d]=0$$  for all $p\in \O$. This implies that there exists an imaginary unit $J$
		such that $c,d\in \spc_J$.
		Using  $(a,b)\perp (c,d)$, we obtain that 
		$$a\overline{c}+b\overline{d}=0.$$
		Since $a\in \R$, $c,d\in \spc_J=\R+\R J$, it follows that $$b\in \spc_J.$$
		Hence $$((a,b), (c,d)) \text{ is an orthonormal basis of } { \mathbb C}_J^2 .$$
		It follows that 
		$$\left(
		\begin{array}[c]{ccc}
			a& b\\
			c& d
		\end{array}\right)\in U_{\spc_J}(2).$$
		
		Suppose $a\notin \R$. In view of Theorem \ref{thm:pxi is wonb}, we get
		$$\frac{\overline{a}}{\abs{a}}\left(
		\begin{array}[c]{ccc}
			a& b\\
			c& d
		\end{array}\right)=:\left(
		\begin{array}[c]{ccc}
			\abs{a}& b'\\
			c'& d'
		\end{array}\right)\in \text{Iso}_\spo(2).$$ 
		By the discussions above, there exists a unit imaginary $J\in \S^6$ such that
		$$\left(
		\begin{array}[c]{ccc}
			\abs{a}& b'\\
			c'& d'
		\end{array}\right)\in U_{\spc_J}(2).$$
		Hence 
		$$\left(
		\begin{array}[c]{ccc}
			a& b\\
			c& d
		\end{array}\right)=\frac{a}{\abs{a}}\left(
		\begin{array}[c]{ccc}
			\abs{a}& b'\\
			c'& d'
		\end{array}\right).$$
		
		Conversely, for any $p\in \S^7, J\in \S^6$ and $\left(
		\begin{array}[c]{ccc}
			a& b\\
			c& d
		\end{array}\right)\in U_{\spc_J}(2)$, we shall show that 
		$$p\left(
		\begin{array}[c]{ccc}
			a& b\\
			c& d
		\end{array}\right)=\left(
		\begin{array}[c]{ccc}
			pa& pb\\
			pc& pd
		\end{array}\right)\in \text{Iso}_\spo(2).$$
		It is easy to check that $((a,b),(c,d))$ is a weak associative orthonormal basis since the sub-algebra of $\O$ generated by two elements is associative.
		It follows from Theorem \ref{thm:pxi is wonb} that 
		$$p\left(
		\begin{array}[c]{ccc}
			a& b\\
			c& d
		\end{array}\right)\in \text{Iso}_\spo(2).$$
		We define
		\begin{eqnarray*}
			\phi:\text{Iso}_\spo(2)&\to& \bigcup_{J\in \S^6}\S^7\times U_{\spc_J}(2)/ \sim\\
			\left(
			\begin{array}[c]{ccc}
				a& b\\
				c& d
			\end{array}\right)&\mapsto&\begin{cases}
				\left[1,\left(
				\begin{array}[c]{ccc}
					a& b\\
					c& d
				\end{array}\right)\right]& a\in \R,\\
				\left[\frac{a}{\abs{a}},\left(
				\begin{array}[c]{ccc}
					\abs{a}& \frac{\overline{a}}{\abs{a}}b\\
					\frac{\overline{a}}{\abs{a}}c& \frac{\overline{a}}{\abs{a}}d
				\end{array}\right)\right]&a\notin \R.
			\end{cases}
		\end{eqnarray*}
		Here $[p,\pmb{T}] $ denotes the  equivalence class  of $(p,\pmb{T})$ in $\S^7\times U_{\spc_J}(2)/ \sim$. One can check this is a bijection.
	\end{proof}
	
	\section{Note on the James question and octonionic Stiefel spaces}
	
	In this section, we give some remarks on octonionic Stiefel spaces and James questions. 
	
	In 1958, I. M. James \cite{james1958stiefel} defined the octonionic Stiefel space $V_k(\O^n)$ as  the space of orthonormal $k$-frames in $\O^n$, i.e.,
	\begin{eqnarray}\label{eqdef:James}
		V_k(\O^n):=\{(x_1,\dots, x_k)\mid x_i\in \O^n,\fx{x_i}{x_j}_{\O}=\delta_{ij}, 1\leqslant i,j\leqslant k\}.
	\end{eqnarray}
	James gave a preliminary study on $V_k(\O^n)$ for $k=2$ and provided two questions on general octonionic Stiefel space:
	\begin{enumerate}
		\item When is the canonical
		projection
		{$$\pi: V_k(\O^n)\to V_l(\O^n)  \qquad   (1\leqslant l< k\leqslant n)$$}
		a fiber map?
		\item  Is $	V_k(\O^n)$ a manifold?
	\end{enumerate}

	%
	
	However,   in the sense of James given by \eqref{eqdef:James}, there may exist the space $V_3(\O^2)$ and  $V_4(\O^2)$, examples can be found in \cite{huoqinghai2021tensor,qianchao2022}. We	
	think it is natural to define the \textbf{octonionic Stiefel spaces} as
	\begin{eqnarray*}
		V^{\text{w}}_k(\O^n)
		:=\{(x_1,\dots, x_k)\mid x_i\in \O^n,(x_1,\dots, x_k) \text{ is a weak associative orthonormal set}\}.
	\end{eqnarray*}
	Since the number of a weak associative orthonormal basis of $\O^n$ must be $n$, it follows that the space $	V^{\text{w}}_k(\O^n)$  does not exist for $k>n$. 	The \textbf{James Questions} are canonically modified as
	\begin{enumerate}
		\item\label{it:ques1} When is the canonical
		projection
		$$\pi: V^{\text{w}}_k(\O^n)\to V^{\text{w}}_l(\O^n), (1\leqslant l< k\leqslant n)$$
		a fiber map?
		\item \label{it:ques2} Is $	V^{\text{w}}_k(\O^n)$ a manifold?
	\end{enumerate}

	Clearly, $V^{\text{w}}_1(\O^n)$ is the sphere $\S^{8n-1}$.
	However, $V^{\text{w}}_2(\O^n)$ is somewhat complicated. It turns out that the canonical projection
	$$\pi:V^{\text{w}}_2(\O^n)\to V^{\text{w}}_1(\O^n)$$ is not a fiber map for $n>1$.
	This is because,  for each $y\in V^{\text{w}}_1(\O^n)$, $\pi ^{-1}(y)$ depends on the real dimension of the real subspace $$\O(\O y):=\left\{\sum_{i=1}^n p_i(r_iy): \ p_i,r_i\in \O, 1\leqslant i\leqslant n, n\in \mathbb{N} \right \},$$
	which is not   constant. Indeed, if $x\in \pi ^{-1}(y)$, then $(x,y)$ is a weak associative orthonormal set. Hence  $\fx{x}{py}=0$ for all $p\in \O$. This means $x\in (\O y)^{\perp}$. It follows from identity \eqref{eq:<>=sum <>R}  that $$(\O y)^{\perp}=(\O(\O y))^{\perp_\R}.$$ Hence $\pi ^{-1}(y)$ is the unitization of $(\O(\O y))^{\perp_\R}$.  If $y\in \R^n$, then $(\O(\O y))=\O y$ and thus $\dim_\R(\O(\O y))=8$. If $y=\dfrac{1}{\sqrt{2}}(1,e_1,0,\dots,0)$ , then
	$$(e_2e_4)y-e_2(e_4y)=\dfrac{1}{\sqrt{2}}(0,2e_7,0,\dots,0)\in (\O(\O y)).$$ Thus there are 8 real linear-independent vector 
	$$y,e_1y,\dots, e_7y,\dfrac{1}{\sqrt{2}}(0,2e_7,0,\dots,0)\in (\O(\O y)).$$ Thus $\dim_\R(\O(\O y))\geqslant 9$. This shows that the projection $\pi$ can not be a fiber map.
	
	\bigskip
	
	James also asked that whether $V_n(\O^n)$ has any properties of a group-space. As demonstrated by Theorem \ref{thm:iso(2)},  we know that $	V^{\text{w}}_2(\O^2)=\text{Iso}_\spo(2)$ is not a group. And in some sense it has a ``book structure",  the spine of which is $\S^7\times O(2)/\sim$.  We remark that $\S^7\times O(2)/\sim$ can be endowed with a Moufang loop structure via
	$$[p,\pmb{T}][q,\pmb{S}]:=[pq,\pmb{TS}],
	$$	
	where $[p,\pmb{T}],[q,\pmb{S}]\in \S^7\times O(2)/\sim$. The loop $\S^7\times O(2)/\sim$ acts on  $	V^{\text{w}}_2(\O^2)$ naturally in a ``alternative" manner, which means the action satisfies:
	$$a(ax)=a^2x, (ax)a=a(xa), x(a^2)=(xa)a$$
	for all $a\in \S^7\times O(2)/\sim$ and $x\in 	V^{\text{w}}_2(\O^2)$. Every page of the book  $	V^{\text{w}}_2(\O^2)$ is $\S^7\times U_{\spc_J}(2)/ \sim$, which also posses a canonical Moufang loop structure.
	
	We think it maybe helpful to study  $V^{\text{w}}_k(\O^n) $ from the viewpoint of loop structure. The development of the theory of octonionic Stiefel space may also contribute to the $G_2$ theory. 
	\bigskip

	\textbf{Acknowledgements} 
	
	The authors are greatly indebted to the anonymous referee for his/her valuable
	comments which improve significantly the presentation of the paper.

	\bibliographystyle{plain}
	
	\bibliography{hilbanach}
\end{document}